\documentclass[12pt]{article}
\usepackage[english]{babel}
\usepackage{amsmath,amsthm}
\usepackage{amsfonts}
\usepackage{graphicx}
\usepackage{enumerate}
\usepackage[usenames,dvipsnames]{color}
\usepackage{subfigure}
\usepackage[right,pagewise,displaymath, mathlines]{lineno}
\usepackage{epstopdf}
\usepackage{color}
\usepackage[applemac]{inputenc} 
\usepackage{dsfont}
\usepackage{mathtools}
\usepackage{float}

\newcommand*\xbar[1]{%
  \hbox{%
    \vbox{%
      \hrule height 0.5pt 
      \kern0.5ex
      \hbox{%
        \kern-0.1em
        \ensuremath{#1}%
      }%
    }%
  }%
} 
\def\inprobHIGH{\,{\buildrel p \over \longrightarrow}\,} 
 
\def\inprob{\,{\inprobHIGH}\,} 
\def\asconv{\,{\buildrel a.s. \over \longrightarrow}\,} 
\def\rhotoZero{\,{\buildrel{\rho \to 0} \over \longrightarrow }\,} 
 
\newcommand{\E}{\ensuremath{\mathbb{E}}}

\newcommand{\Prob}{\ensuremath{\mathbb{P}}}
\numberwithin{equation}{section}
\numberwithin{figure}{section}
\numberwithin{table}{section}

\def\cN{\mathcal{N}}
\def\cF{\mathcal{F}}
\newtheorem{theorem}{Theorem}[section]
\newtheorem{corollary}[theorem]{Corollary}
\newtheorem{lemma}[theorem]{Lemma}

\numberwithin{equation}{section}

\begin{document}


\begin{center}
{\Large 
Inferential results\\for a new measure of inequality}
\vspace*{8mm}

{\large
Youri Davydov$^{\textrm{a}}$,
Francesca Greselin$^{\textrm{b}}$}
\bigskip

$^{\textrm{a}}$\textit{Laboratoire  Paul Painlevé, Université des Sciences et Technologies de Lille (Lille 1), Lille, France \\and Saint Petersburg State University, Saint Petersburg, Russia} 
\smallskip

$^{\textrm{b}}$\textit{Dipartimento di Statistica e Metodi Quantitativi, Università di Milano Bicocca, 
Milan, Italy}
\smallskip

\end{center}

\medskip

\begin{quote}
\noindent
\textbf{Abstract.} In this paper we derive inferential results for a new index of inequality, specifically defined for capturing significant  changes  observed  both in the left and in the right tail of the income distributions. The latter shifts are an apparent fact for many countries like US, Germany, UK, and France in the last decades, and are a concern for many policy makers.
We propose two empirical estimators for the index, and show that they are asymptotically equivalent. Afterwards, we adopt one estimator and  prove its consistency and  asymptotic normality. Finally we introduce an empirical estimator for its variance and provide conditions  to show its convergence to the finite theoretical value.
An analysis of real data on net income from the Bank of Italy Survey of Income and Wealth is also presented, on the base of the obtained inferential results.
\medskip

\end{quote}

\section{Introduction}

In view of measuring economic inequality in a society, suppose that we are interested, for instance, in incomes. Let $X$ be an 'income' random variable with non negatively supported cdf $F(x)$. 

Next, define $Q(p)=F^{-1}(p)=\inf\{x \, : \, F(x) \geq p,\;\; p \in [0,1]\}$ as the $p$-th quantile  of $X$, and suppose that $X$ possesses a finite  mean 
\begin{equation*}
\mu_F =\int_0^\infty xF(dx)=\int_0^1 F^{-1}(p)dp.
\end{equation*}
The Lorenz curve, introduced by Lorenz (1905), is an irreplaceable tool in this domain. It is
 defined by 
\begin{equation}
l_F (p)=\frac{1}{ \mu_F }\int_0^p F^{-1}(t)dt.
\label{lor}
\end{equation}
The curve 
$l_F (p)$ expresses the share of income possessed by the $p$\% poorer part of  population. It has been expressed firstly by Pietra (1915, with English translation now available as Pietra, 2014), and mathematically formulated as in (\ref{lor}) by Gastwirth  (1971).

In the following we will also employ 
$m_F (p)=1-l_F (1-p),$ which provides the share of income owned by the richer $p$\% of the population.  
Obviously, 
$m_F (p)$ is the curve obtained by applying  a central symmetry to 
 $l_F (p)$, with respect to the center of the unit square, as shown in Figure \ref{fig:MeL} and allows us also to rephrase the Gini into  
 $G_F=\int_0^1 (m_F (p)-l_F (p))dp$. 
\begin{figure}[h!]
\centering
\includegraphics[height=7cm,width=7cm]{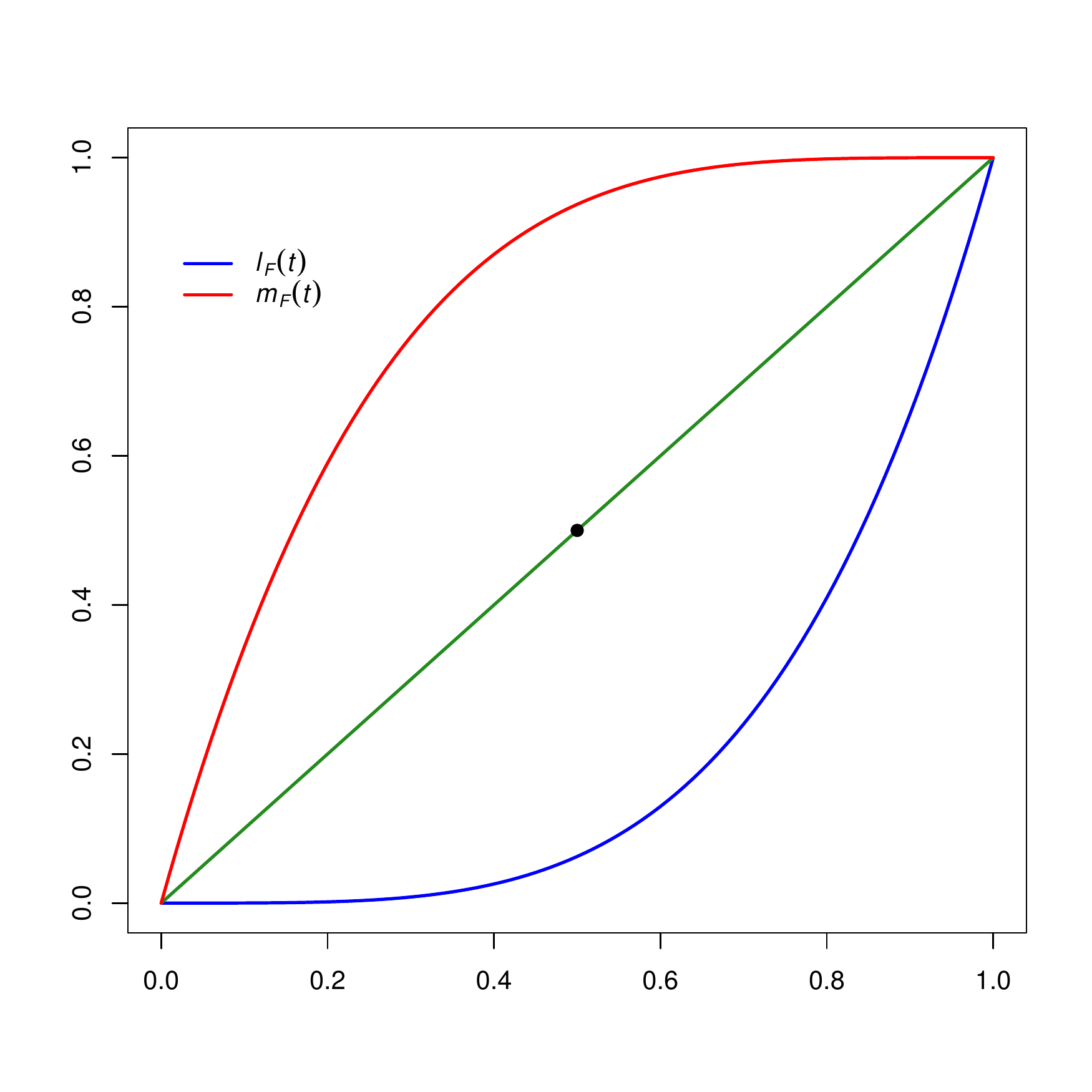}
\caption{\footnotesize{Curves $l_F (p)$ and $m_F (p)$}}
\label{fig:MeL}
\end{figure}
We recall that the Gini can be rephrased as the weighted average of all comparisons made among the mean income of the poorest and the overall mean (Greselin \textit{et al.} 2012, Greselin 2014). When dealing with skewed distributions, as it is the case for many economic size distributions, the median should be preferred to the mean, in such a way that Gastwirth (2014) proposed to modify the Gini accordingly. \\
Very recently, motivated by the observed shifts toward the extreme values in income distributions,  a new focus is introduced in Gastwirth (2016), almost contemporarily to Davydov and Greselin (2016). Policy makers  are nowadays  interested in understanding what happens in the more critical portions of the population, as significant  changes have been observed  both in the left and in the right tail of the income distributions in countries like US, Germany, UK, France in the last decades. Notice that the classical Lorenz curve provides useful pointwise information with reference to poorest people, while on the other hand, as $L(p)$ approaches 1 as $p$ approaches 1, it does not display the variation within the upper end (f.i., top 5\% or 1\%) of the distribution clearly. 
 The novel approach is to  consider equally sized and  opposite
groups  of  population, and  compare their mean income. Aiming at  contrasting the economic position of the group of the poorer 
$p$\%  to the one of the $p$\% of the richest, the following  inequality curve has been introduced
\begin{equation}
D_F (p)=
\frac{m_F (p)-l_F (p)}{m_F (p)}.
\label{DavydovCurve}
\end{equation}
In the case of perfect equality, each fraction $p$ of population has same mean income, hence $D_F (p)=0$ for all $p \in [0,1]$. While the income distribution moves toward more variability, the mean income 
$\frac{\mu_F m_F (p)}{p}$ of the $p\%$ of richest people will be moving far from  the mean income 
$\frac{\mu_F l_F (p)}{p}$ of the $p\%$ of poorest part of the population, and $D_F(p)$ raises toward $1$. Hence,  we can represent the pointwise measure of inequality in the population by  plotting $D_F(p)$.  

Naturally, we can summarize all the information given by the curve $D_F(p)$ in a single measure of inequality $D_F$, by taking the expected value 
\begin{equation}
D_F=\int_0^1 D_F (p)dp.
\label{Davydov}
\end{equation}
Notice that $D_F$ is the area between the observed inequality curve $D_F (p)$ and the curve of perfect equality, which is the horizontal line passing through the origin of the axes.

The structure of the paper is as follows. Section \ref{estimators} introduces two estimators for the new inequality measure, and provide reasons for selecting them in view of their main purpose. 
The third section, which is the core of the paper, states the main inferential results, 
in more detail we will show the consistency of the estimators, state their asymptotical distribution, and the asymptotic negligibility of their difference. We also introduce an empirical estimator for the variance, and establish its convergence to the finite variance of the estimator.
Some lemmas useful for the inferential theory have been presented  in Section \ref{ProofOfLemmas}, along with their proof.
Section \ref{application} shows how the inferential results can be employed to develop an analysis on real income data. Final considerations are given in Section \ref{concRemarks}.

\section{Estimators}
\label{estimators}
Economic data on the entire (or complete) population is rarely available, so most studies are based on data obtained from well-designed sample surveys. Hence we usually have to estimate summary measures from samples. We  introduce here two empirical estimators, say $ \widehat{D}_n$ and $\widetilde{D}_n$ for estimating $D_F$. The first one is derived, in a very natural way, by replacing  the population cdf $F(x)$ and mean value $\mu_F$ in (\ref{Davydov})  by their empirical counterparts $F_n(x)=\frac{1}{n} \sum_{i=1}^n \mathds{1}_{[0,x]} \big( X_i \big)$, and $\xbar{X}_n=\frac{1}{n} \sum_{i=1}^n X_i$, and
then considering the empirical Lorenz curve, say 
 $$l_n(p)= \frac{1}{\xbar{X}_n} \int_0^p F_n^{-1}(s)ds, \quad \textrm{and its dual } \quad
 m_n(p)=\frac{1}{\xbar{X}_n}\int_{1-p}^1 F_n^{-1}(s)ds,$$ as follows
\begin{align}
\widehat{D}_n&=\displaystyle{\int_0^1\displaystyle{ \frac{m_n(p)-l_n(p)}{m_n(p)} dp}}  \label{firstEstD}\\
&:=1-  \displaystyle{\int_0^1 G_n(p) dp} \nonumber 
\end{align}
where we set 
$G_n(p)={l_n(p)}/{m_n(p)}$.

Then, the second estimator is defined in terms of   the order statistics  $X_{1:n}\le \cdots \le X_{n:n}$   of the i.i.d. sample   $X_{1},X_{2},...,X_{n}$ drawn from $X$, 
therefore we define
\begin{align}
\widetilde{D}_n&=1- \frac{1}{n} \displaystyle\sum_{i=1}^{n} \frac{\sum_{k=1}^{i}X_{k:n}}{\sum_{k=n-i+1}^{n}X_{k:n}} \nonumber \\
&:=1- \frac{1}{n} \sum_{i=1}^{n} G_n(i/n)
\label{secondEstD}
\end{align}
where $ G_n(i/n)$ expresses the ratio between the mean income of the poorest $i$ and of the richest $i$ elements in the sample. 

We will show later, in Theorem \ref{Th3}, that the two estimators   $\widehat{D}_n$ and  $\widetilde{D}_n$ are asymptotically equivalent. While the estimator $\widehat{D}_n$ is  suitable   for developing  inferential results,  $\widetilde{D}_n$ is much simpler when it comes to implement code for the analysis of real data.

\section{Inferential results}
\label{main}
In this Section we will present our main results, starting from the consistency of the estimator $\widehat{D}_n$, next we state its asymptotic normal distribution, and then we deal with its variance estimation. 
Finally, we will show the  asymptotic equivalence of the two estimators $\widetilde{D}_n$ and $\widehat{D}_n$. 

Unless explicitly stated otherwise, we assume throughout that the cdf $F(x)$ of $X$ is a continuous function. This is a natural choice when modeling income or wealth distributions, and for many other economic size distributions.

\subsection{Consistency of  $\widehat{D}_n$}

\begin{theorem}
 $\widehat{D}_n$ is a consistent estimator for $D_F$.
\end{theorem}
\begin{proof}
From the normalized definition of the empirical Lorenz curve and its dual, say  $l_n(p)$ and $m_n(p)$, it is useful here to introduce their absolute versions, given by $L_n(p)=\int_0^p F_n^{-1}(s)ds$ and  $M_n(p)=\int_{1-p}^1 F_n^{-1}(s)ds$. 
We may rephrase $\widehat{D}_n$ as
\begin{align}
\widehat{D}_n&=\displaystyle{\int_0^1\displaystyle{ \frac{M_n(p)-L_n(p)}{M_n(p)} dp}} \label{firstEstDII}
\end{align}

For all $p \in [0,1]$, we have that 
$L_n(p)$ converges, with probability 1, uniformly to 
$L(p)=\int_0^p F^{-1}(s)ds$ (see Goldie, 1977). With the same approach, we have that $M_n(p)$ converges, with probability 1, uniformly to $M(p)=\int_{1-p}^1 F^{-1}(s)ds$. As $L(p) \leq M(p)$, due to the Lebesgue dominated convergence theorem we get the thesis.
\end{proof}
 
\subsection{Asympthotic normality of the estimator $\widehat{D}_n$}
 
\begin{theorem}
If the moment $\mathbb{E}|X|^{2+\delta}$ is finite for some $\delta >0$, then we have the asymptotic representation
\begin{eqnarray}
\sqrt{n} \left( \widehat{D}_n - D \right) = \frac{1}{\sqrt{n}} \sum_{i=1}^n h(X_i)+o_{\textbf{P}}(1),
\label{linearization}
\end{eqnarray}
where $o_{\textbf{P}}(1)$ denotes a random variable that converges to 0 in probability when $n \to \infty$, and 
\begin{eqnarray*}
h(X_i)= \int_0^{+\infty} \left[ \mathds{1}_{[0,x]} \left( X_i \right)- F(x) \right] \omega \big( F(x) \big) dx 
\end{eqnarray*}
with the weight function $\omega(t)=\omega_1(t)+\omega_2(t)$, where 
\begin{equation}
\omega_1(t)=\int_t^1 \frac{1}{M(s)}  \, ds \quad \textrm{and} \quad \omega_2(t)=\int_0^t \frac{L(1-s)}{\left[M(1-s)\right]^2}  \,ds.
\label{omega1e2}
\end{equation}

\label{Th1}
\end{theorem}

\begin{corollary}
Under the conditions of Theorem \ref{Th1}, we have that  $\widehat{D}_n$ is asymptotically normally distributed, that is
\begin{eqnarray*}
\sqrt{n} \left( \widehat{D}_n - D \right)  \underset{n\rightarrow\infty}{\Longrightarrow} \cN(0,\sigma_F^2),
\end{eqnarray*}
where 
\begin{equation*}
\sigma_F^2 =\int_0^\infty \left[ \int_0^y F(x)\,\omega\big(F(x)\big) \, dx \right] \big(1-F(y)\big) \, \omega\big(F(y)\big) \, dy.
\end{equation*}

\label{Co1}
\end{corollary}
The proof follows immediately from \eqref{linearization} by applying the Central Limit Theorem of P. L\'{e}vy.

\textit{Proof of Theorem \ref{Th1}}\\
From the  definition of $\widehat{D}_n$ and $D$, we get
\begin{align}
\sqrt{n} \left( \widehat{D}_n - D \right) &=-\sqrt{n} \displaystyle{\int_0^1 \left[ \frac{L_n(t)}{M_n(t)} - \frac{L(t)}{M(t)} \right] dt} \nonumber \\[7pt] 
&=-\sqrt{n} \displaystyle{\int_0^1 \frac{L_n(t)-L(t)}{M(t)}dt  + \sqrt{n}\int_0^1 \frac{L(t)\left[M_n(t)-M(t) \right]}{[M(t)]^2} dt + R_n(t)}  
\label{linear_Dn}
\end{align}
where the remainder term is given by $R_n(t)=R_n^{(1)}+R_n^{(2)}$ and
\begin{align}
R_n^{(1)}&=  \sqrt{n} \int_0^1 \big( L_n(t)-L(t) \big) \left( \frac{1}{M(t)} - \frac{1}{M_n(t)} \right) dt, \label{Rn1}\\[7pt] 
R_n^{(2)}&= \sqrt{n} \int_0^1  \frac{L(t)}{M(t)}  \left( \frac{1}{M_n(t)} - \frac{1}{M(t)} \right) \big( M_n(t) - M(t) \big) dt.\label{Rn2}
\end{align}
 We will later show  (Lemma \ref{Lemma1} and  \ref{Lemma2}, respectively) that $R_n^{(1)}$ and $R_n^{(2)}$ are of order $o_{\textbf{P}}(1).$
The proof follows the approach of 
Greselin, Pasquazzi and Zitikis (2010), to state the asymptotic normality for the Zenga inequality  index (Zenga, 2007). Hence we now proceed our analysis of the first two terms in (\ref{linear_Dn}), by using the Vervaat process 
\begin{eqnarray}
V_n(p)= \int_0^p \big(F_n^{-1}(t)-F^{-1}(t) \big) dt +  \int_0^{F^{-1}(p)} \big(F_n(x)-F(x) \big) dx 
\end{eqnarray}
and its dual, 
\begin{eqnarray}
V^*_n(p)= \int_p^1 \big(F_n^{-1}(t)-F^{-1}(t) \big) dt +  \int_{F^{-1}(p)}^{\infty} \big(F_n(x)-F(x) \big) dx 
\end{eqnarray}for which we know that $V^*_n(p)=-V_n(p)$. For mathematical and historical details on the Vervaat process, see Zitikis (1998), Davydov and Zitikis (2004), and Greselin et al. (2009). Now, denoting the uniform on $[0,1]$ empirical process by  $e_n(p)=  \sqrt{n}(F_n(F^{-1}(p))-p)$ and using one property of the Vervaat process, namely
\begin{eqnarray}
 \sqrt{n} V_n(p) \leq  | e_n(p) | |F_n^{-1}(p)-F^{-1}(p) |, 
 \label{Vervaat1}
\end{eqnarray}
we find a bound for the first term in (\ref{linear_Dn}) as follows
\begin{equation*}
-\sqrt{n} \displaystyle{\int_0^1 \frac{L_n(t)-L(t)}{M(t)} \, \, dt      =       \sqrt{n}  \int_0^1  \frac{1}{M(t)} \left[ \int_0^{F^{-1}(t)} \big( F_n(x)-F(x) \big) dx \right] dt +\mathcal{O}(R_{n}^{(3)})}
\end{equation*}
where 
\begin{align*}
R_n^{(3)}&= - \sqrt{n} \int_0^1 \frac{1}{M(t)}\,V_n(t) \,\, dt  \\[7pt] 
&\leq \int_0^1  \frac{1}{M(t)} | e_n(t)| \left| F_n^{-1}(t)- F^{-1}(t) \right| dt. 
\end{align*}
We will later show  (Lemma \ref{Lemma3}) that $R_n^{(3)}=o_{\textbf{P}}(1).$

Now we deal with the second term in  (\ref{linear_Dn}), and we obtain, using similar arguments as before
\begin{align*}
\sqrt{n}\int_0^1 &\displaystyle{\frac{L(t)\left[M_n(t)-M(t) \right]}{\left[M(t)\right]^2} \,\, dt }\nonumber \\[7pt] 
&= - \sqrt{n}  \int_0^1  \frac{L(t)}{\left[ M(t) \right]^2} \left[ \int_{F^{-1}(1-t)}^{+\infty} \big( F_n(x)-F(x) \big) dx \right] dt +\mathcal{O}(R_{n}^4),\nonumber 
\end{align*}
where
\begin{align*}
R_n^{(4)}&= \sqrt{n} \int_0^1 V^*_n(1-t) \frac{L(t)}{\left[M(t)\right]^2}\,\, dt  \\[7pt] 
&\leq \int_0^1 | e_n(t)| \left| F_n^{-1}(t)- F^{-1}(t) \right| \frac{L(t)}{\left[M(t)\right]^2}\,\, dt. 
\end{align*}
In Lemma \ref{Lemma4} we show that $R_n^{(4)}=o_{\textbf{P}}(1),$
therefore we have
\begin{align}
\sqrt{n} \left( \widehat{D}_n - D \right)&=\displaystyle{\sqrt{n}  \int_0^1  \frac{1}{M(t)} \left[ \int_0^{F^{-1}(t)} \big( F_n(x)-F(x) \big) dx \right] dt }   \nonumber \\
&\displaystyle{- \sqrt{n}  \int_0^1  \frac{L(t)}{\left[ M(t) \right]^2} \left[ \int_{F^{-1}(1-t)}^{+\infty} \big( F_n(x)-F(x) \big) dx \right] dt +o_{\textbf{P}}(1). } 
\label{linear_Dn1}
\end{align}

We notice that the first term  in the right hand side of equation (\ref{linear_Dn1}) could be rewritten as 
\begin{align*}
\sqrt{n}  \int_0^1  \frac{1}{M(t)}&\displaystyle{ \left[ \int_0^{F^{-1}(t)} \big( F_n(x)-F(x) \big) dx \right] dt } \nonumber \\
&= \displaystyle{\sqrt{n}  \int_0^1 \left( \int_0^{+\infty} \big( F_n(x)-F(x) \big) \mathds{1}_{[0,F^{-1}(t)]}(x)dx \right)  \frac{1}{M(t)} dt} \nonumber \\
&= \displaystyle{\sqrt{n}  \int_0^{+\infty}  \left( \int_0^1 \big( F_n(x)-F(x) \big) \mathds{1}_{[F(x),1]}(t)\frac{1}{M(t)} dt  \right)  dx} \nonumber \\
&= \displaystyle{\sqrt{n}  \int_0^{+\infty}   \big( F_n(x)-F(x) \big)  \left( \int_{F(x)}^1 \frac{1}{M(t)} dt  \right)  dx} \nonumber \\
&= \displaystyle{\frac{1}{\sqrt{n} } \sum_{i=1}^n h_1(X_i)}
\end{align*}
where 
$$ h_1\left( X_i \right)= \int_0^{+\infty} \left[ \mathds{1}_{[0,x]} \left( X_i \right)- F(x) \right] \omega_1\big( F(x) \big) dx,$$
and 
$$ \omega_1\left( t \right) =\int_t^1 \frac{1}{M(s)}ds.$$
For the second term  in the right hand side of equation (\ref{linear_Dn1}), using the change of variable  $t=1-s$, we obtain:
\begin{equation*}
- \displaystyle{\sqrt{n}  \int_0^1  \frac{L(1-s)}{ \left[ M(1-s) \right] ^2} \left[ \int_{F^{-1}(s)}^{+\infty} \big( F_n(x)-F(x) \big) dx \right] ds }=- \displaystyle{\frac{1}{\sqrt{n} } \sum_{i=1}^n h_2(X_i)}, 
\end{equation*}
where
$$ h_2\left( X_i \right)= \int_0^{+\infty} \left[ \mathds{1}_{[0,x]} \left( X_i \right)- F(x) \right] \omega_2\big( F(x) \big) dx,$$
and 
$$ \omega_2\left( t \right) =\int_0^t \frac{L(1-s)}{\left[ M(1-s) \right]^2}ds=\int_{1-t}^1 \frac{L(s)}{\left[ M(s) \right]^2 }ds.$$
This completes the proof of Theorem \ref{Th1}.
\qed
 
\subsection{Convergence of the empirical variance}
We deal here with the theoretical variance $Var\big(h(X_1)\big)$, that is
\begin{equation}
\sigma_F^2 =\int_0^\infty \left[ \int_0^y F(x)\,\omega\big(F(x)\big) \, dx \right] \big(1-F(y)\big) \, \omega\big(F(y)\big) \, dy 
\label{sigma_F}
\end{equation}
and  its empirical counterpart
\begin{equation}
\sigma_n^2 =\int_0^\infty \left[ \int_0^y F_n(x) \, \omega\big(F_n(y)\big) \, dx \right] \big(1-F_n(y)\big) \, \omega\big(F_n(y)\big) \,dy.
\label{sigma_n}
\end{equation}
Let $x_0 \geq 0$ be the minimum value in the support of $F(x)$, i.e. the value such that $F(y)=0$ for $y<x_0$ and $F(y)>0$ if $y>x_0$. Analogously, let $T_0$ be maximum value in the support of $F(x)$, i.e. such that $F(x) <1$ $\forall x < T_0$ and $F(T_0)=1$. Notice that we may have $T_0=+\infty$. 
 Then we have
\begin{itemize}
\item $F_n(x)=0 \quad \forall  x < x_0$ because $X \geq x_0$ a.s.,
\item $\sigma_F^2 =\int_{x_0}^\infty \left[ \ldots \right] \ldots dy$,
\item $\sigma_n^2 =\int_{x_0}^\infty \left[ \ldots\right] \ldots dy$.
\end{itemize}
Therefore, without loss of generality, we can take $x_0=0$.

\begin{theorem}
Assume that $\mathbb{E} |X|^{2+\delta} < +\infty$ for some $\delta>0$. Then,  we have a.s.
\begin{equation}
\sigma_n^{2} \,\,\underset{n\rightarrow\infty}{\longrightarrow}  \,\, \sigma_F^{2}.
\label{sigman}
\end{equation}
\label{Th2}
\end{theorem}

\begin{proof}
The proof is composed by   three steps.

\textit{Step 1}: \\
For all $\epsilon, T$ such that $0<\epsilon<T<T_0$, we will show that, with probability 1, for almost all $y \in [\epsilon, T]$, and with $\omega(t)=\omega_1(t)+\omega_2(t)$ given by \eqref{omega1e2}  we have 
\begin{equation}
\int_0^y F_n(x) \, \omega\big(F_n(y)\big)  \,dx \xrightarrow[n\to\infty]{}  \int_0^y F(x) \,\omega(F(x))  \, dx.
\label{step1a}
\end{equation}

%
We begin by the study of the first part of (\ref{step1a}) related to $\omega_1(t)$, i.e.
\begin{equation*}
\int_0^y F_n(x)  \,\omega_1(F_n(x))  \, dx= \int_0^y \int_0^1  \mathds{1}_{[F_n(x),1]}(s)F_n(s) \frac{1}{M_n(s)} \, ds \, dx := \int_0^y \int_0^1  \psi_n(x,s)  \, ds \, dx.
\end{equation*}
We know that 
\begin{itemize}
\item as  $F_n(x) \to F(x)$ a.s. (uniformly), we have the convergence  $\mathds{1}_{[F_n(x),1]}(s)  \to  \mathds{1}_{[F(x),1]}(s)$ with probability 1  for almost all $s \in [0,1]$ and $x \in R^+$,
\item $\forall s\in[0,1]$ and with probability 1, we have that $$M_n(s)=\int_{1-s}^1 F_n^{-1}(y)dy \qquad \xrightarrow[n\to\infty]{} \qquad \int_{1-s}^1 F^{-1}(y)dy =M(s),$$
\item $M_n(s) \geq s \, \xbar{X}_n\quad \forall s \in[0,1]$.
\end{itemize}
As $\xbar{X}_n \to \mu_F$ a.s., with probability 1 there exists a constant $C>0$ such that 
\begin{equation*}
  \psi_n(x,s)  \leq \mathds{1}_{[F_n(x),1]}(s)F_n(x) \frac{1}{s \, \xbar{X}_n}  \leq C.
\end{equation*}
 Hence, Lebesgue theorem  gives (\ref{step1a}), with $\omega$ replaced by $\omega_1$.

Now we consider the second part of (\ref{step1a}), where $\omega_2(t)$ takes the place of $\omega$:
\begin{equation*}
\int_0^y F_n(x) \, \omega_2(F_n(x)) \,dx= \int_0^y \int_0^1  \mathds{1}_{[0,F_n(x)]}(s)F_n(x) \frac{L_n(1-s)}{\left[M_n(s) \right]^2} \, ds \, dx := \int_0^y \int_0^1  \widetilde{\psi}_n(x,s)  \, ds \, dx 
\end{equation*}
and observe that
\begin{itemize}
\item $\mathds{1}_{[0,F_n(x)]}(s) \to \mathds{1}_{[0,F_(x)]}(s)$ with probability 1  for almost all $s \in [0,1]$ and $x \in R^+$,
\item $\forall s \in [0,1]$, with probability 1, we have that $L_n(1-s) \to L(1-s)$, and $M_n(1-s) \to M(1-s)$, 
\item $L_n(1-s) \leq M(1-s) \quad \forall s \in[0,1]$.
\end{itemize}
Therefore
\begin{eqnarray*}
\widetilde{\psi}_n(x,s) &\leq&  \mathds{1}_{[0,F_n(x)]}(s)F_n(x) \, \frac{1}{M_n(1-s)}  \\[5pt]
&\leq&  \mathds{1}_{[0,F_n(x)]}(s)F_n(x) \, \frac{1}{(1-s) \, \xbar{X}_n}\\[5pt]
&\leq& \frac{F_n(x)} {\big(1-F_n(x)\big) \xbar{X}_n}   \quad \leq   \quad\frac{1}{\big(1-F_n(T)\big) \xbar{X}_n} \quad \to  \quad \frac{1}{\big(1-F(T)\big) \mu_F} < +\infty.
\end{eqnarray*}
Once more using Lebesgue theorem we get $\forall y \leq T$
\begin{equation*}
\int_0^y F_n(x)  \, \omega_2(F_n(x))  \, dx \quad \asconv \quad \int_0^y F(x) \, \omega_2(F(x))  \, dx
\end{equation*}
which completes the proof of (\ref{step1a}).

\textit{Step 2}: \\
For all  $\epsilon, T$ such that  $0 < \epsilon < T < T_0$ and given
\begin{equation*}
\sigma_n^2 [\epsilon,T]= \int_\epsilon^T \Psi_n(y) dy 
\end{equation*}
where
\begin{equation*}
\Psi_n(y)= \left[ \int_0^y F_n(x) \omega \big(F_n(x) \big) dx \right] \big(1-F_n(y)\big)\omega\big(F_n(y)\big),
\end{equation*}
we will show that, with probability 1,
\begin{equation}
\sigma_n^2 [\epsilon,T] \xrightarrow[n\to\infty]{} \sigma^2 [\epsilon,T].
\label{ConvEpsilonT}
\end{equation}
Due to the previous step, for every $y$ we know that $\Psi_n(y)$  converges, with probability 1, to
\begin{equation*}
\Psi(y)= \left[ \int_0^y F(x) \, \omega \big(F(x)\big) \, dx \right] \big(1-F(y)\big) \, \omega\big(F(y)\big).
\end{equation*}
We have shown that, with probability 1, 
\begin{equation*}
\limsup \limits_{n} \left( \sup_{\epsilon \leq x \leq T} F_n(x) \, \omega\big(F_n(x)\big) \right) \leq \left(1+ \frac{1}{1-F(T)}\right) \frac{1}{\mu_F},
\end{equation*}
and using
\begin{equation*}
\omega_1\big(F_n(y)\big)=\int_{F_n(y)}^1 \frac{1}{M_n(s)} \, ds \leq \frac{1}{F_n(y) \, \xbar{X}_n}, 
\end{equation*}
it follows that we have a.s.
\begin{equation*}
\limsup \limits_{n} \left( \sup_{\epsilon \leq y \leq T} \omega_1\big(F_n(y)\big)\right) \leq \frac{1}{F(\epsilon) \mu_F}.
\end{equation*}
Hence 
\begin{equation}
\limsup\limits_{n} \left( \sup_{\epsilon \leq x \leq T} \Psi_n(y)  \right) \leq y \left( 1+ \frac{1}{1-F(T)}\right) \left( \frac{1}{F(\epsilon)}+ \frac{1}{1-F(T)}\right)\frac{1}{\mu_F^2}.
\label{boundForLimSup}
\end{equation}
Observing now that the function at the right hand side in (\ref{boundForLimSup}) is integrable on $[\epsilon, T]$, we can
apply Lebesgue's dominated convergence theorem and prove (\ref{ConvEpsilonT}).

\textit{Step 3}: \\
To complete the proof of Theorem \ref{Th2} we need to obtain 
a bound for  the integrals $\sigma_n^2 [0, \epsilon]$ and $\sigma_n^2 [T, \infty]$. We use the following more delicate estimation of $\omega(t)$, for all $\gamma >0$: there exists a positive constant  $C_\gamma$ such that 
\begin{equation}
\omega(t) \leq \frac{C_\gamma}{\mu_F} \,\,\frac{1}{t^\gamma (1-t)^\gamma}.
\label{omegaBound}
\end{equation}
Indeed, as  $L(t) \leq t \mu_F \leq  M(t)$, for  $t \in [0,1]$,
we have 
\begin{equation*}
\omega_1(t) \leq \frac{1}{\mu_F} \,\,\int_t^1 \frac{1}{s}ds=\frac{1}{\mu_F}  | \ln t | .
\end{equation*}
and 
\begin{equation*}
\omega_2(t) \leq \frac{1}{\mu_F} \,\,\int_0^t \frac{1}{1-s}ds=\frac{1}{\mu_F}  | \ln (1-t )| .
\end{equation*}
Hence, for every $\gamma >0$, there exists a constant  $0<C_\gamma<+\infty$, depending only on $\gamma$, such that for all $t \in (0,1)$
\begin{equation*}
\omega_1(t) \leq \frac{C_\gamma}{\mu_F\, t^\gamma},  \quad \omega_2(t) \leq \frac{C_\gamma}{\mu_F\, (1-t)^\gamma}, 
\end{equation*}
which jointly give the estimation \eqref{omegaBound}.

For $\gamma \leq 1/2$, we have
\begin{eqnarray}
\sigma_n^2 [0, \epsilon]&=&\int_0^\epsilon \left[ \int_0^y F_n(x) \, \omega \big(F_n(x)\big) \, dx \right] \big(1-F_n(y)\big) \, \omega\big(F_n(y)\big) \, dy  \nonumber \\[7pt]
     &\leq & \int_0^\epsilon \left[ \int_0^y F_n(x)  \frac{C_\gamma}{F_n^\gamma(x)\big(1-F_n(x)\big)^\gamma  \, \xbar{X}_n} \, dx \right] \big(1-F_n(y) \big) \frac{C_\gamma}{F_n^\gamma(x)\big(1-F_n(x)\big)^\gamma \, \xbar{X}_n} \, dy  \nonumber \\[7pt]
     &\leq & \frac{C_\gamma^2}{\xbar{X}_n^2} \int_0^\epsilon \left[ \int_0^y F_n(x)^{1-\gamma} dx \right]  \, \big(1-F_n(y)\big)^{1-2 \gamma} \, F_n(y)^{-\gamma} \, dy \nonumber \\[7pt]
     &\leq & \frac{C_\gamma^2}{\xbar{X}_n^2}\int_0^\epsilon  y \, F_n^{1-2 \gamma} (y) \,  \big(1-F_n(y)\big)^{1-2\gamma} \, dy \nonumber \\[7pt]
     &\leq & \frac{C_\gamma^2}{\xbar{X}_n^2}\int_0^\epsilon  y \, dy  \nonumber \\[7pt]
     & =& \frac{\epsilon^ 2 \,C_\gamma^2}{2 \,\,\xbar{X}_n^2}.
     \label{bound0epsilonEmp}
\end{eqnarray}
We have similarly
\begin{eqnarray}
\sigma_n^2 [T, \infty] & \leq & \frac{C_\gamma^2}{\xbar{X}_n^2} \int_T^\infty y F_n^{1-2\gamma}(y) \big(1-F_n(y)\big)^{1-2\gamma} dy  \nonumber \\
 & \leq &   \frac{C_\gamma^2}{\xbar{X}_n^2} \int_T^\infty y \big(1-F_n(y)\big)^{1-2\gamma} dy.
\end{eqnarray}

Let us introduce a new probability space $(\widetilde{\Omega},\widetilde{\cF},\widetilde{\Prob})$ and a new random variable $Y$, taking the values   $X_i$, for $i=1, \ldots, n$, such that  $\widetilde{\Prob}\big(Y=X_i\big)=\frac{1}{n}$. Then
\begin{eqnarray*}
1-F_n(y)&=& \widetilde{\Prob} \left\{ Y > y \right\} \nonumber \\[7pt]
           &\leq & \frac{\widetilde{\E} |Y|^{2+p}}{y^{2+p}}\nonumber \\[7pt]
           &=& \frac{1}{y^{2+p}}\left\{\frac{1}{n} \sum_{i=1}^n X_i^{2+p} \right\}.
\end{eqnarray*}
If $p \leq \delta$ then, due to the strong law of large numbers,
\begin{eqnarray*}
\frac{1}{n} \sum_{i=1}^n X_i^{2+p} \quad \asconv \quad \E |X|^{2+p} < +\infty,
\end{eqnarray*}
hence, with probability 1, we have
\begin{eqnarray}
\limsup\limits_{n}\, \sigma_n^2 [T, \infty] &\leq&  \frac{C_\gamma^2}{\mu_F^2}  \,\, \E(X^{2+\delta}) \int_T^\infty y^{(1-(1-2\gamma)(2+\delta))} dy.
\label{limsup}
\end{eqnarray}

Observing now that, for $\gamma <\delta/[2(2+\delta)]$, the integral $\int_T^\infty y^{\left(1-(1-2\gamma)(2+\delta)\right)} dy $ converges, and
defining  $\beta$ such that $\big(1-(1-2\gamma)(2+\delta)\big)=-(1+\beta)$,  then (\ref{limsup}) takes the form
\begin{eqnarray}
\limsup\limits_{n}  \sigma_n^2 [T, \infty]  &\leq&   \frac{C_\gamma^2}{\mu_F^2} \,\, \E(X^{2+\delta}) \, \frac{1}{\beta \,T^\beta}.
\label{boundTinftyEmp}
\end{eqnarray}

Evidently, replacing $F_n(x)$ by $F(x)$ in (\ref{bound0epsilonEmp}) and (\ref{boundTinftyEmp}), we obtain their theoretical counterparts 
\begin{eqnarray}
\sigma_F^2 [0, \epsilon]&\leq & \frac{\epsilon^2}{2}  \,\,\frac{C_\gamma^2}{\mu_F^2}, 
\label{bound0epsilonTeo}
\end{eqnarray}
and 
\begin{eqnarray}
\sigma_F^2 [T, \infty]  &\leq&    \frac{C_\gamma^2}{\mu_F^2}  \,\, \E(X^{2+\delta}) \, \frac{1}{\beta \,T^\beta}.
\label{boundTinftyTeo}
\end{eqnarray}

Now, collecting the bounds (\ref{bound0epsilonEmp}) and (\ref{bound0epsilonTeo}), the convergence stated in (\ref{ConvEpsilonT}), and finally bounds (\ref{boundTinftyEmp}) and (\ref{boundTinftyTeo}) from the three steps
\begin{align}
&\limsup\limits_{n} \left| \sigma_n^2-\sigma_F^2 \right| \nonumber  \\
&\leq  \limsup\limits_{n}\left| \sigma_n^2 [0, \epsilon]-\sigma_F^2 [0, \epsilon] \right|+\limsup\limits_{n} \left| \sigma_n^2[\epsilon, T] - \sigma_F^2[\epsilon, T] \right|+    \limsup\limits_{n}  \left| \sigma_n^2 [T, \infty] -\sigma_F^2 [T, \infty]  \right|  \nonumber\\
 &\leq \epsilon^2 \frac{C_\gamma^2}{\mu_F^2}+ \frac{2}{\beta}  \frac{C_\gamma^2}{\mu_F^2} \,\, \E(X^{2+\delta}) \,\,\frac{1}{T^\beta}.
 \label{finalbound}
\end{align}
 Taking $\epsilon \to 0$ and $T \to \infty$ in  (\ref{finalbound}), we arrive at (\ref{sigman}).
\end{proof}

Having established the consistency and asymptotic normality for  the estimator $\widehat{D}_n$, we would like to prove similar properties  for the second estimator  $\widetilde{D}_n$ defined in (\ref{secondEstD}). To do this, we will focus on their difference $\Delta_n:=\widehat{D}_n - \widetilde{D}_n$ and prove its asymptotic negligibility.  
 
\begin{theorem}
If the moment $\mathbb{E}|X|^{2+\delta}$ is finite for some $\delta >0$, then we have  
\begin{equation}
\sqrt{n} |\Delta_n | \inprob 0.
\end{equation}
\label{Th3}
\end{theorem}
Before proving Theorem \ref{Th3}, it is worth to state two useful Corollaries.
\begin{corollary}

If the moment $\mathbb{E}|X|^{2+\delta}$ is finite for some $\delta >0$, then we have
\begin{equation}
\sqrt{n}(\widetilde{D}_{n}-D_{F}) \Longrightarrow \mathcal{N}(0, \sigma_F^2)
\label{star-1}
\end{equation}
where $\sigma_F^2 =Var\big(h(X_1)\big)$ is the theoretical variance.
\end{corollary}

\begin{corollary}

Under the same assumptions, we have also
\begin{equation}
\sqrt{n}\frac{(\widetilde{D}_{n}-D_{F})}{\sigma_n} \Longrightarrow \mathcal{N}(0, 1)
\label{star-2}
\end{equation}
where $\sigma_n^2$ is the empirical counterpart  for $\sigma_F^2$ given by \eqref{sigma_n}. 

The same is true if we replace $\widetilde{D}_{n}$ by $\widehat{D}_{n}$.
\end{corollary}
 
\begin{proof}(of Theorem \ref{Th3}).
Let $\epsilon:=\epsilon_n=\frac{m_n}{n}\sim n^{-\alpha-1/2}$, where $0<\alpha<\frac{\delta}{2(2+\delta)}$. We have
\begin{eqnarray}
| \Delta_n | &\leq& \left| \int_0^\epsilon G_n(t)dt-\frac{1}{n} \displaystyle{\sum_{i=1}^{m_n} G_n (i/n)} \right| +\left| \int_\epsilon^1 G_n(t)dt-\frac{1}{n} \displaystyle{\sum_{i=m_n+1}^{n} G_n (i/n)} \right|  \nonumber \\
&\leq& 2 \epsilon + \left| \displaystyle{\sum_{i=m_n+1}^{n} \int_{(i-1)/n}^{i/n} \left[G_n (t)-G_n (i/n) \right] dt} \right|  \nonumber\\
\vspace{2mm}  \nonumber\\
&\leq& 2 \epsilon +  \displaystyle{\omega_{G_n}^{[\epsilon,1]}(1/n)}   \label{DeltanBound}     
\end{eqnarray}
where $\omega_{G_n}^{[\epsilon,1]}$ is the modulus of continuity of $G_n$ on the interval $[\epsilon, 1]$ given by
\begin{equation*}
\omega_{G_n}^{[\epsilon,1]}(h)=\sup_{\epsilon \leq t, \, s \leq 1, \, |t-s| \leq h} |G_n (t)-G_n (s)|.
\end{equation*}

Let $ t \in \left[(i-1)/n, i/n \right]$, then
\begin{eqnarray}
\left| G_n (t)-G_n (i/n) \right| &=& \displaystyle{\left| \frac{L_n(t)}{M_n(t)}- \frac{L_n(i/n)}{M_n(i/n)} \right|} \nonumber \\[5pt]
                                             & \leq & \frac{1}{M_n(t)} \left| L_n(t)-L_n(i/n) \right| + \frac {L_n(i/n)}{M_n(t) M_n(i/n)}\left| M_n(i/n)-M_n(t) \right| \nonumber  \\[7pt]
                                             & \leq & \frac{2}{t \,\, \xbar{X}_n} \, \omega_{L_n}^{[\epsilon,1]}(1/n),      \label{GnBound}                               
\end{eqnarray}
where we used $M_n(s) =1-L_n(1-s)$ and the inequalities $L_n(s) \leq M_n(s) $, $  
s \,\,  \xbar{X}_n \leq M_n(s) $ that hold true $ \forall s \in [0,1]$.
From the bounds in (\ref{DeltanBound}) and (\ref{GnBound})   we get
\begin{equation}
| \Delta_n | \leq 2 \epsilon +\frac{2}{\epsilon \, \xbar{X}_n}  \, \omega_{L_n}^{[\epsilon,1]}(1/n).
\label{Deltan}
\end{equation}
As for $t \in [(i-1)/n, i/n]$ 
\begin{equation*}
\left| L_n (t)-L_n (i/n) \right| = \int_t^{i/n} F_n^{-1}(s) ds =X_{i:n}(i/n-t) \leq \frac{1}{n}X_{i:n} ,   
\end{equation*}
we get
 \begin{equation*}
 \omega_{L_n}^{[\epsilon,1]}(1/n) \leq \frac{1}{n}\, \, \, \max_{m_n \leq k \leq n} X_{k:n} \leq \frac{1}{n} \, M_n \qquad \textrm{where} \qquad M_n =\max_{ k \leq n} X_{k:n}.
 \end{equation*}
 Therefore, due to (\ref{Deltan}), we obtain
 \begin{equation*}
\sqrt{n}| \Delta_n | \leq \frac{2}{n^\alpha} + \frac{2 M_n}{ \xbar{X}_n \,n^{1/2-\alpha}}.
 \end{equation*}
As $2/n^\alpha \rightarrow 0$ and $ \xbar{X}_n \rightarrow \mu_F$, it is sufficient to state the convergence in probability of $ M_n/n^{1/2 -\alpha}$ to 0.
We have, for $t>0:$ 
 \begin{eqnarray*}
\Prob \left\{ \frac{M_n}{n^{1/2 -\alpha} } \geq t \right\} &=& \Prob \left\{M_n \geq t n^{1/2-\alpha} \right\} 
\leq n \, \Prob \left\{ X \geq t n^{1/2-\alpha} \right\} \\
&\leq& n \frac{\E|X|^{2+\delta}}{(tn^{1/2-\alpha})^{2+\delta}} \xrightarrow[n\to\infty]{} 0 
 \end{eqnarray*}
as we have chosen $	\alpha < \frac{\delta}{2(2+\delta)}$. 
\end{proof}

\section{Proofs}
\label{ProofOfLemmas}
\begin{lemma}
\label{Lemma1}
Under the conditions of Theorem \ref{Th1}, we have that $$R_n^{(1)}=o_{\textbf{P}}(1).$$
\end{lemma}

\begin{proof}
We estimate  $R_n^{(1)}$ by splitting the integral in two parts, by choosing $\rho \in (0,1)$
 \begin{align}
 R_n^{(1)}&=\sqrt{n} \int_0^1 \big( L_n(t)-L(t) \big) \left( \frac{1}{M(t)} - \frac{1}{M_n(t)} \right)  dt \nonumber \\
  &=\sqrt{n} \int_0^\rho \ldots dt +\sqrt{n} \int_\rho^1\ldots dt \nonumber \\
  &:=R_n^{(1,1)}+R_n^{(1,2)}.
\label{eq1lemma1}
\end{align}
We now look for getting a bound for $R_n^{(1,1)}$ and  initially deal with its first part, given by
 \begin{eqnarray}
 \sqrt{n} \int_0^\rho  \big( L_n(t)-L(t) \big)  \frac{1}{M(t)} \,  dt.  
  \label{firstR_n11}
 \end{eqnarray}
Provided that $t<\rho$, we have that $M(t) \geq t \,F^{-1}(1-\rho)$, and  
 \begin{eqnarray}
&& \sqrt{n} \int_0^\rho \frac{1}{t} \big| L_n(t)-L(t) \big| dt \nonumber \\
&&=\sqrt{n} \displaystyle{ \int_0^\rho \frac{1}{t} \bigg\{ V_n(t) -\int_0^{F^{-1}(t)}  \big[ F_n(x)-F(x) \big] dx  \bigg\} dt }\nonumber \\
&&=\sqrt{n} \displaystyle{ \int_0^\rho \frac{1}{t}  V_n(t) dt -\sqrt{n} \int_0^\rho \frac{1}{t} \bigg\{\int_0^{F^{-1}(t)}  \big[ F_n(x)-F(x) \big] dx  \bigg\} dt. }
\label{eq2lemma1}
\end{eqnarray}
Now we consider  the lefthand term in (\ref{eq2lemma1}), and set
\begin{equation}
K_n=\sup_{t \in (0,1)}  \,\,\,\ \frac{ \left| e_n(t)\right| } {t^{1/2-\epsilon}(1-t)^{1/2-\epsilon}},
\label{K_n}
\end{equation}
where $e_n(t)=\sqrt{n} \big( F_n(F^{-1}(t))-t \big)$. We know that (see (9.2) in Greselin \textit{et al.} 2010)
\begin{equation}
K_n=O_{\textbf{P}}(1).
\end{equation}
Therefore,
employing the inequality in (\ref{Vervaat1}) related to the Vervaat process, and choosing $\epsilon$ such that $0<\epsilon < \frac{1}{2}$, we obtain 
 \begin{align*}
 \sqrt{n} \displaystyle{ \int_0^\rho \frac{1}{t}  V_n(t) dt }&\leq  \displaystyle{ \int_0^\rho \frac{1}{t} | e_n(t) | |F_n^{-1}(t)-F^{-1}(t) | \, dt}\\
&\leq K_n \displaystyle{\int_0^{\rho}  \frac{1}{t}   t^{1/2-\epsilon}(1-t)^{1/2-\epsilon}  \left| F_n^{-1}(t)-F^{-1}(t) \right|  dt}\\
&\leq K_n \displaystyle{\int_0^{\rho}  \frac{1}{t^{1/2+\epsilon}} \left| F_n^{-1}(t)-F^{-1}(t) \right|  dt}.
  \end{align*}
As $| F_n^{-1}(t)-F^{-1}(t) | \leq F_n^{-1}(\rho)+F^{-1}(\rho) \asconv 2 \, F^{-1}(\rho)$, by Lebesgue dominated convergence theorem the integral
  \begin{equation*}
   \displaystyle{\int_0^{\rho}  \frac{1}{t^{1/2+\epsilon}} \left| F_n^{-1}(t)-F^{-1}(t) \right|  dt \asconv 0} \quad \textrm{as} \quad n \to \infty.
  \end{equation*}
Hence 
 \begin{align*}
 \sqrt{n} \displaystyle{ \int_0^\rho \frac{1}{t}  V_n(t) dt = o_{\textbf{P}}(1). }
 \end{align*} 
The righthand term in (\ref{eq2lemma1}) may be estimated as follows
 \begin{eqnarray}
&& \displaystyle{\int_0^\rho \frac{1}{t} \bigg\{\int_0^{F^{-1}(t)} \sqrt{n} \big| F_n(x)-F(x) \big| dx  \bigg\} dt } \nonumber \\
&& \leq K_n \displaystyle{\int_0^\rho \frac{1}{t} \bigg\{\int_0^{F^{-1}(t)}  F(x)^{1/2-\epsilon}\big(1-F(x)\big)^{1/2-\epsilon}  dx  \bigg\} dt }  \nonumber \\
&& \leq  K_n \displaystyle{\int_0^\rho \frac{1}{t}  \,  \,t^{1/2-\epsilon}\bigg\{\int_0^{+\infty}  \big(1-F(x)\big)^{1/2-\epsilon}  dx  \bigg\} dt } \nonumber \\
&& \leq  C_1 K_n \displaystyle{\int_0^\rho \ t^{-1/2-\epsilon}dt =C_1 K_n \phi(\rho)}
 \label{boundRHTed2lemma1}
\end{eqnarray}
for all $0<\epsilon<1/2$, where we set $ \phi(\rho)=\displaystyle{\int_0^\rho \ t^{-1/2-\epsilon}dt}$,  and $C_1=\int_0^{+\infty}  \big(1-F(x)\big)^{1/2-\epsilon}  dx <+\infty$. The latter quantity is finite, due to the existence of the $2+\delta$ moment of $X$. 

To complete the analysis of $R_n^{(1,1)}$ we have to deal now with its second part, given by
\begin{eqnarray}
 \sqrt{n} \int_0^\rho  \big( L_n(t)-L(t) \big)  \frac{1}{M_n(t)} \,  dt.
 \label{secondR_n11}
 \end{eqnarray}
As $M_n(t) \geq t F_n^{-1}(1-\rho)$ for $t \in [0,1]$ and $F_n^{-1} (1-\rho) \asconv F^{-1} (1-\rho)$, the bound for   \eqref{secondR_n11} can be found by following the same steps as for \eqref{firstR_n11}. 

We continue our proof now by finding a bound for the second term $R_n^{(1,2)}$ in (\ref{eq1lemma1}).
As $\rho\leq t\leq 1$, then 
 \begin{equation*}
M(t) \geq \int_{1-\rho}^{1} F^{-1}(s)ds >0 
\end{equation*} 
and 
 \begin{equation*}
M_n(t) \geq \int_{1-\rho}^{1} F_n^{-1}(s)ds \geq \int_{1-\rho}^{1} F^{-1}(s)ds+ o_{\textbf{P}}(1).
\end{equation*}
Therefore, setting $H_n:=\displaystyle{\sup_{s \in [1-\rho,1]} \,\,{\frac{1}{M_n(s) \, M(s)}}}=
\mathcal{O}_{\textbf{P}}(1)$,  we have
\begin{eqnarray}
&&\sqrt{n} \int_\rho^1  \big| L_n(t)-L(t) \big| \big| M_n(t)-M(t) \big| dt  \nonumber\\
&&\leq H_n \sqrt{n} \int_\rho^1  \big| L_n(t)-L(t) \big| \big| M_n(t)-M(t) \big| dt.
\label{eq3lemma1}
\end{eqnarray}
We observe that (\ref{eq3lemma1}) is $o_{\textbf{P}}(1)$ if the following two equalities hold true:
\begin{equation}
\sqrt{n} \, \, \big| \xbar{X}_n-\mu_F \big| \int_{\rho}^{1} \big| L_n(t) -L(t) \big| \, dt=o_{\textbf{P}}(1),
\label{eq4lemma1}
\end{equation}
and 
\begin{equation}
\sqrt{n} \int_{\rho}^{1} \big| L_n(t) -L(t) \big| \, \,\big| L_n(1-t) -L(1-t) \big| \, \, dt=o_{\textbf{P}}(1),
\label{eq5lemma1}
\end{equation}
by recalling that
 \begin{equation}
M_n(t)-M(t) = (\xbar{X}_n-\mu_F)-[L_n(1-t)-L(1-t)] 
\label{equivMandL}
\end{equation}
due to 
\begin{equation*}
\xbar{X}_n-\mu_F= \int_{0}^{1-t} \big[F_n^{-1}(s) -F^{-1}(s)\big] ds + \int_{1-t}^1 \big[F_n^{-1}(s) -F^{-1}(s)\big] ds.
\end{equation*}

To get (\ref{eq4lemma1}), remark that $\sqrt{n} \, \, \big| \xbar{X}_n-\mu_F \big| =O_{\textbf{P}}(1)$, and 
\begin{eqnarray*}
&& \int_{\rho}^{1} \big| L_n(t) -L(t) \big| \, dt  \\
 && \leq \int_{0}^{1}  \int_{0}^{t} \big| F^{-1}_n(s) -F^{-1}(s) \big| ds  \, dt  \\
 && \leq \int_{0}^{1} \big| F^{-1}_n(s) -F^{-1}(s) \big| \, ds =o_{\textbf{P}}(1).
\end{eqnarray*}
Finally, to get (\ref{eq5lemma1}), we begin with the inequality
\begin{eqnarray*}
&&\sqrt{n} \, \, \int_{\rho}^{1} \big| L_n(t) -L(t) \big| \, \,\big| L_n(1-t) -L(1-t) \big| \, \, dt \\
 && \leq \sqrt{n} \, \, \int_{\rho}^{1} \big[ L_n(t) -L(t) \big]^2 \, \,dt,
 \end{eqnarray*}
and use the following bound for the latter integrand 
\begin{eqnarray*}
\sqrt{n} \big[ L_n(t) -L(t) \big]^2 \leq 2 \sqrt{n} \big[V_n(t)\big]^2+2 \sqrt{n} \bigg\{ \int_0^{F^{-1}(t)} \big| F_n(x)-F(x) \big| dx \bigg\}^2.
\end{eqnarray*}
Recalling that 
\begin{eqnarray*}
\sqrt{n} \, \,\int_{\rho}^{1}  \big[V_n(t)\big]^2 \,  dt=o_{\textbf{P}}(1),
\end{eqnarray*}
and exploiting  \eqref{K_n} we get 
\begin{eqnarray*}
&&\sqrt{n} \, \,\bigg\{ \int_0^{F^{-1}(t)} \big| F_n(x)-F(x) \big| dx \bigg\}^2 \\
&&=\frac{1}{\sqrt{n}} \, \,\bigg\{ \int_0^{F^{-1}(t)} \sqrt{n} \, \, \big| F_n(x)-F(x) \big| dx \bigg\}^2 \\
&&\leq K_n \frac{1}{\sqrt{n}} \, \,\bigg\{ \int_0^{F^{-1}(t)}    F(x)^{1/2-\epsilon} \big(1-F(x)\big)^{1/2-\epsilon}  dx \bigg\}^2\\
&& \leq K_n  \frac{1}{\sqrt{n}} \Big( \int_0^{+\infty} \big(1-F(x)\big)^{1/2-\epsilon} dx \Big)^2=o_{\textbf{P}}(1).
\end{eqnarray*}
Integrating in $dt$ on $[\rho,1]$ we hence obtain the desired bound.

From the previous estimates it follows that $\forall \rho: \, \,0 <\rho <1$ we have
\begin{equation*}
 R_n^{(1)}=  U_n{(\rho)}+ T_n{(\rho)},
\end{equation*}
where 
\begin{equation*}
U_n{(\rho)}=o_{\textbf{P}}(1), \quad T_n{(\rho)} \leq C_1 \, K_n \, \phi(\rho),  \quad \textrm {with} \quad K_n=O_{\textbf{P}}(1) \quad \textrm {and} \quad \phi(\rho) \,\, \rhotoZero \,0.
\end{equation*}
Fixing $\epsilon > 0$, let $C>0$ be such that $\Prob \left\{ |K_n|>C  \right\} <\epsilon$  $\, \forall n$, and let $\rho_\epsilon >0 $ be such that for  $\rho < \rho_\epsilon$ we have $\phi(\rho)< {\epsilon}/{2 C_1 C}$.
Then, having 
\begin{equation*}
 \Prob \left\{R_n^{(1)}> \epsilon \right\} \leq \Prob \left\{U_n{(\rho)}> \epsilon/2 \right\}+ \Prob \left\{T_n{(\rho)}> \epsilon/2 \right\},
\end{equation*}
we get,  for  $\rho < \rho_\epsilon$,
\begin{equation*}
\limsup \limits_{n}  \Prob \left\{T_n{(\rho)}> \epsilon/2 \right\} \leq  \Prob \left\{ |K_n|> \frac{\epsilon}{2 C_1 \phi(\rho) } \right\} \leq  \Prob \left\{ |K_n|> C  \right\}  < \epsilon,
\end{equation*}
which finally gives  $R_n^{(1)}=o_{\textbf{P}}(1).$
\end{proof}
 \begin{lemma}
\label{Lemma2}
Under the conditions of Theorem \ref{Th1}, we have that
 \begin{eqnarray*}
 R_n^{(2)}=o_{\textbf{P}}(1).
  \end{eqnarray*}
  \end{lemma}
 
\begin{proof}
We start from the definition of $R_n^{(2)}$ in (\ref{Rn2}) here recalled for convenience $$R_n^{(2)}= \sqrt{n} \int_0^1  \frac{L(t)}{M(t)}  \left( \frac{1}{M_n(t)} - \frac{1}{M(t)} \right) \big( M_n(t) - M(t) \big) dt.$$
Observing that  $L(t) \leq M(t)$  for $t \in[0,1]$ and using  (\ref{equivMandL}) to rewrite $\big(M_n(t)-M(t)\big)$, the proof can be established following the proof of  Lemma (\ref{Lemma1}) with minor modifications.
  \end{proof}
 \begin{lemma}
\label{Lemma3}
Under the conditions of Theorem \ref{Th1}, we have that
 \begin{eqnarray*}
 R_n^{(3)}=o_{\textbf{P}}(1).
  \end{eqnarray*}
  \end{lemma}

\begin{proof}
We estimate   $R_n^{(3)}$ by splitting it in two integrals as follows
 \begin{eqnarray}
R_n^{(3)}&=&           \int_0^1  \frac{1}{M(t)} | e_n(t)| \left| F_n^{-1}(t)- F^{-1}(t) \right| dt \nonumber\\
   &=&            \int_0^{1/2} \ldots dt +            \int_{1/2}^1 \ldots dt \nonumber\\
&:=&R_n^{(3,1)}+R_n^{(3,2)}.
   \label{splitRn3}
   \end{eqnarray}
Let us consider the first term $R_n^{(3,1)}$ and observe that 
 \begin{eqnarray*}
 \frac{1}{M(t)} \leq  \frac{1}{F^{-1}(\frac{1}{2})\, \,t}   
 \end{eqnarray*}
 where we assume that $F^{-1}(\frac{1}{2}) >0$, otherwise we may replace $F^{-1}(\frac{1}{2})$ by $F^{-1}(a) >0$, with $a \in (0,1)$ appropriately chosen.
 Hence, by choosing $\epsilon \leq \frac{1}{2}$, and recalling that $e_n(t)=\sqrt{n} \left(F_n\big(F^{-1}(t)\big)-t\right)$, we arrive at 
\begin{eqnarray*}
R_n^{(3,1)} &\leq&  \displaystyle{\frac{1}{F^{-1}(\frac{1}{2})} \int_0^{1/2}  \frac{1}{t}  \frac{  \sqrt{n} \left| F_n(F^{-1}(t))-t \right|  \left| F_n^{-1}(t)-F^{-1}(t) \right|}{t^{1/2-\epsilon}(1-t)^{1/2-\epsilon}}t^{1/2-\epsilon}(1-t)^{1/2-\epsilon} dt}\nonumber\\
&\leq& \displaystyle{ \frac{K_n}{F^{-1}(\frac{1}{2}) }\int_0^{1/2}  \frac{1}{t^{1/2+\epsilon}}  \left| F_n^{-1}(t)-F^{-1}(t) \right| =o_{\textbf{P}}(1).} \nonumber
 \end{eqnarray*}
as $K_n=\mathcal{O}_{\textbf{P}}(1)$ and $F_n^{-1}(t) \asconv  F^{-1}(t)$ for $t \in [0,1]$.

Now we deal with $R_n^{(3,2)}$, i.e. the second term in    (\ref{splitRn3}). Observing that $M(t)=\int_{1-t}^1 F^{-1}(s) ds \geq \int_{1/2}^1 F^{-1}(s) ds=c>0$, we obtain
\begin{eqnarray*}
R_n^{(3,2)} &\leq& \frac{1}{c} \displaystyle{\int_{1/2}^1  | e_n(t)| \left| F_n^{-1}(t)- F^{-1}(t) \right| dt}\nonumber\\
&\leq &\frac{K_n}{c} \displaystyle{\int_{1/2}^1   \left| F_n^{-1}(t)- F^{-1}(t) \right| dt}=o_{\textbf{P}}(1).
  \end{eqnarray*}
 \end{proof}
\begin{lemma}
\label{Lemma4}
Under the conditions of Theorem \ref{Th1}, we have that
$$ R_n^{(4)}= o_{\textbf{P}}(1).$$
 \end{lemma}
\begin{proof}
 We will deal with  $R_n^{(4)}$, as for the previous Lemma, by splitting it as follows 
  \begin{eqnarray}
 R_n^{(4)} &=&\int_0^1   \left| e_n(1-t) \right|  \left| F_n^{-1}(t)- F^{-1}(t) \right| \frac{L(t)}{[M(t)]^2} dt \nonumber\\
 &=&\int_0^{1/2} \ldots dt + \int_{1/2}^1 \ldots dt \nonumber\\
 &:=& R_n^{(4,1)} + R_n^{(4,2)} 
    \label{splitRn4}
    \end{eqnarray}
and we initially consider $R_n^{(4,1)}$. 
Observing that  $M(t) \geq t F^{-1}(1/2)$  for $t <1/2$, we have  
  \begin{eqnarray*}
 R_n^{(4,1)} &\leq&  \frac{K_n 2^{1/2-\epsilon}}{F^{-1}(1/2)} \int_0^1/2  t^{-1/2-\epsilon}   \left| F_n^{-1}(t)- F^{-1}(t) \right| dt      =o_{\textbf{P}}(1).
  \end{eqnarray*}

Finally, the result on $R_n^{(4,2)}$ comes from observing that for $t \in (1/2,1)$ there exists a constant $C$ such that $M(t) \geq C$, and that $\sup_{t \in(0,1)} e_n(1-t) \leq K_n $, we have
\begin{eqnarray*}
 R_n^{(4,2)} \leq \frac{K_n}{C} \int_{1/2}^1  \left| F_n^{-1}(t)- F^{-1}(t) \right| dt =o_{\textbf{P}}(1),
\end{eqnarray*}
due to the assumption on the second (hence the first) moment finite on $X$.
\end{proof}

\section{The new inequality measure on real data}
\label{application}

The purpose of this section is to show, through  a real data application,  the theoretical results obtained in the previous sections. We  employ the Bank of Italy Survey on Household Income and Wealth (hereafter named by its acronym, SHIW) dataset, published in 2016. This survey contains information on household post-tax income and wealth in the year 2014, covering 8,156 households, and 19,366
individuals. The sample is representative of the Italian population, which is composed
of about 24,7 million households and 60,8 million individuals.
The SHIW provides information on each individual's Personal Income Tax net
income, but does not contain the corresponding gross income. We employ an updated version of the microsimulation model described in Morini and Pellegrino (2016) to
estimate the latter for each taxpayer.
A comparison of the results from the microsimulation model with the official statistics
published by the Italian Ministry of Finance (2016) shows that the distribution of gross income
and of net tax, according to bands of gross income and type of employment, are
close to each other. The empirical analysis we develop here is based on the observed net income from the SHIV, while tax data and gross income arise from the microsimulation model. 

To appreciate the asymptotic results of Section \ref{main} on the empirical estimator $\widetilde{D}_n$, we calculate four types of confidence intervals: the normal, the basic, the percentile and the BCa confidence intervals.  After drawing the bootstrap samples, the  empirical estimator is evaluated at each sample, and  Figure \ref{fig:hist} show the histograms of the obtained values when considering Gross Income in panel (a),   Net Income in panel (b) and  Taxes in panel (c). While inequality estimators have a skewed distribution in case of low sample size, here the accuracy of  the normal approximation is apparent, due to the large sample size. As a further check of the quality of the first order approximation,  Figure \ref{fig:Q-Q} shows the Q-Q plots obtained for the three cases.
\begin{figure}[h!]
\centering
\includegraphics[height=5cm,width=5cm]{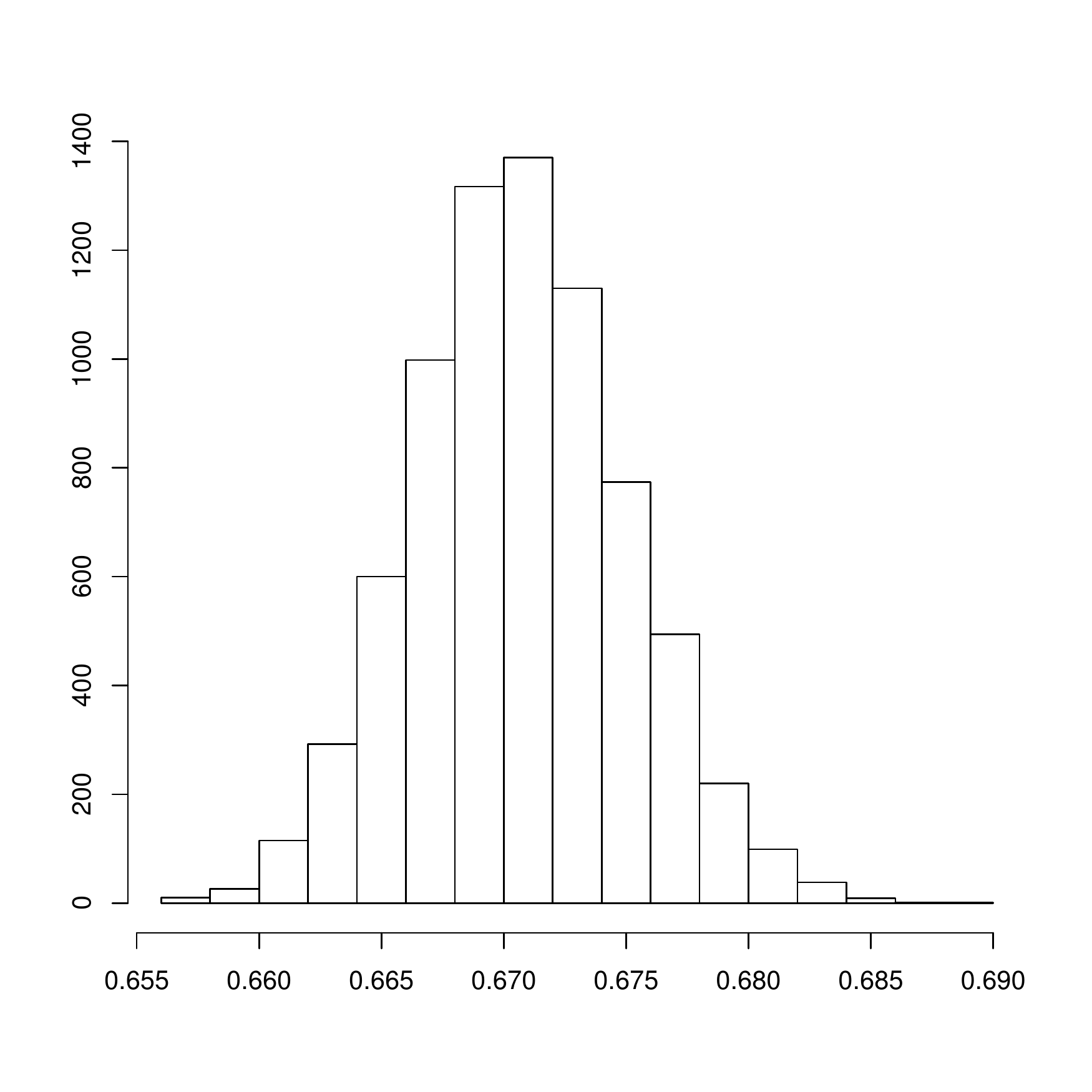}\includegraphics[height=5cm,width=5cm]{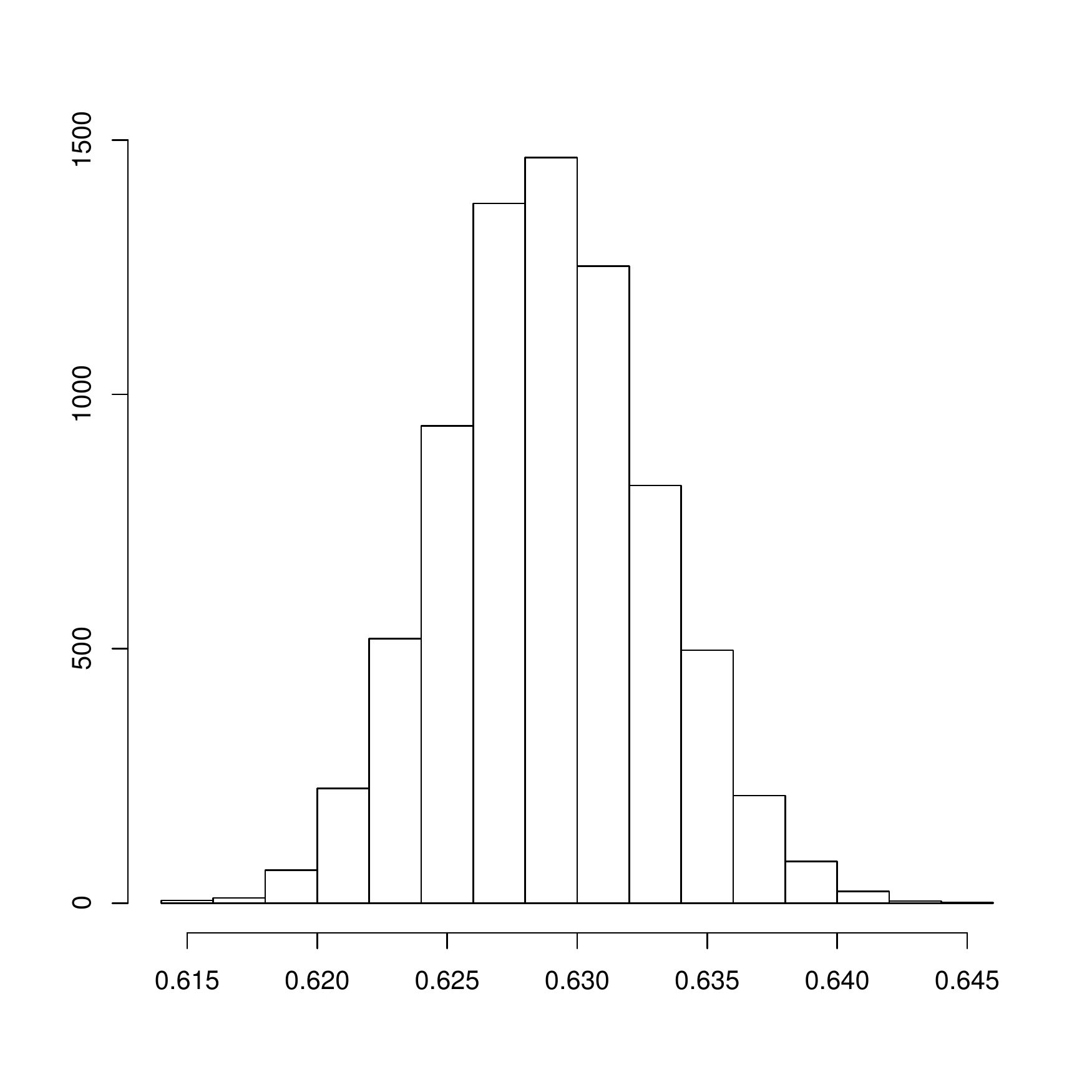}\includegraphics[height=5cm,width=5cm]{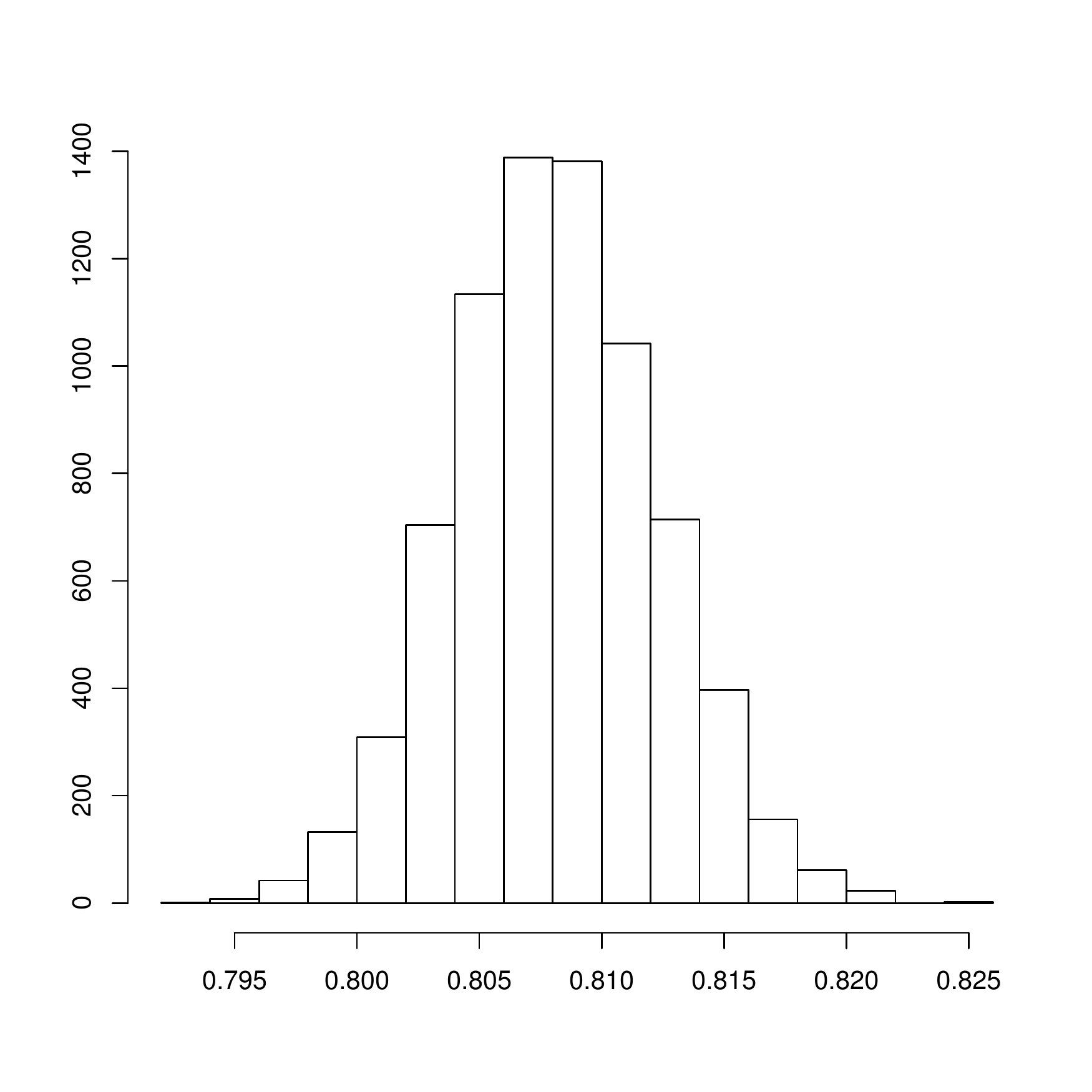}\\
(a)\hspace{4.3cm}  (b) \hspace{4.3cm} (c)
\caption{\footnotesize{Histograms for $\widetilde{D}_n$ on Gross Income in panel (a),  on Net Income in panel (b) and on Taxes in panel (c)}}
\label{fig:hist}
\end{figure}

\begin{figure}[h!]
\centering
\includegraphics[height=5cm,width=5cm]{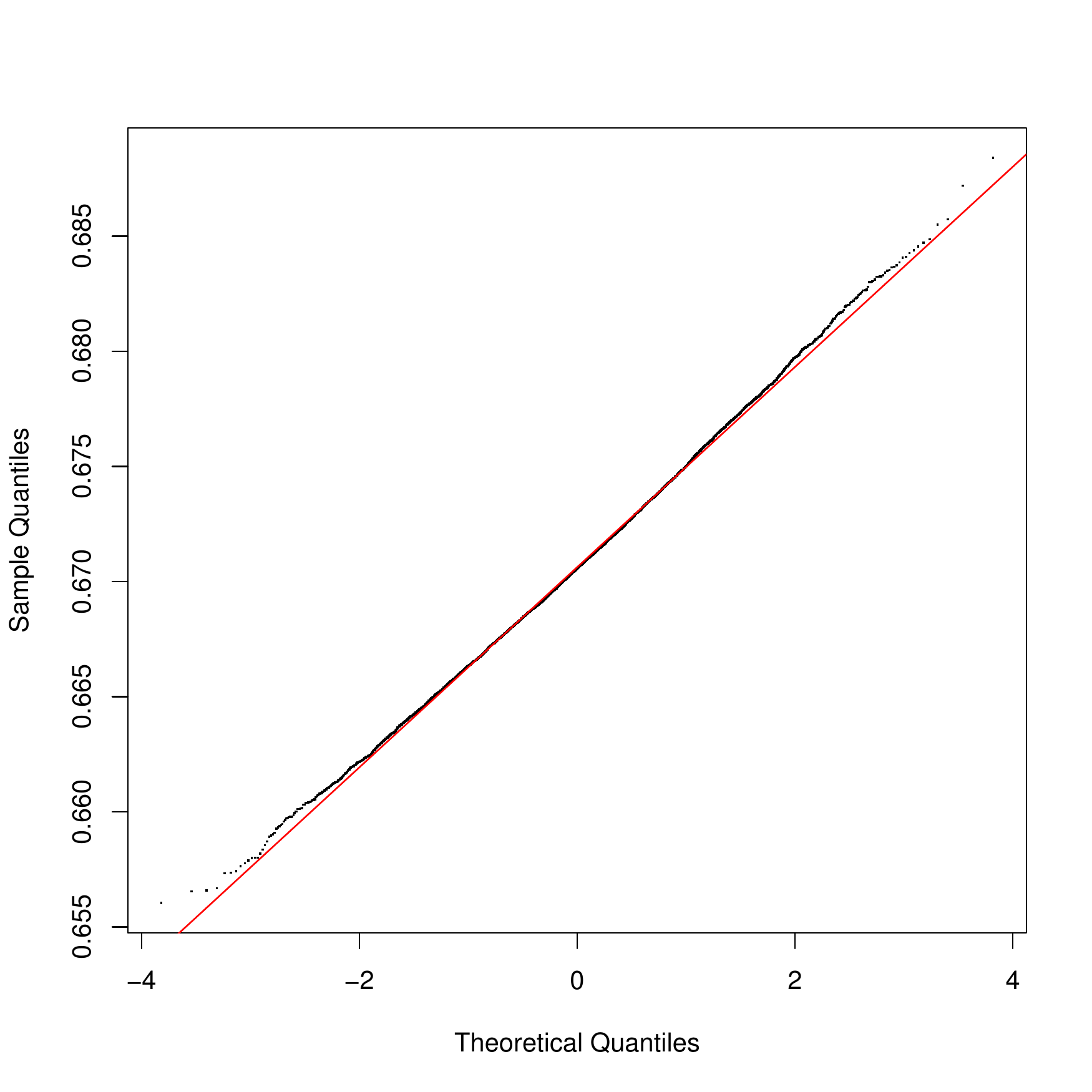}\includegraphics[height=5cm,width=5cm]{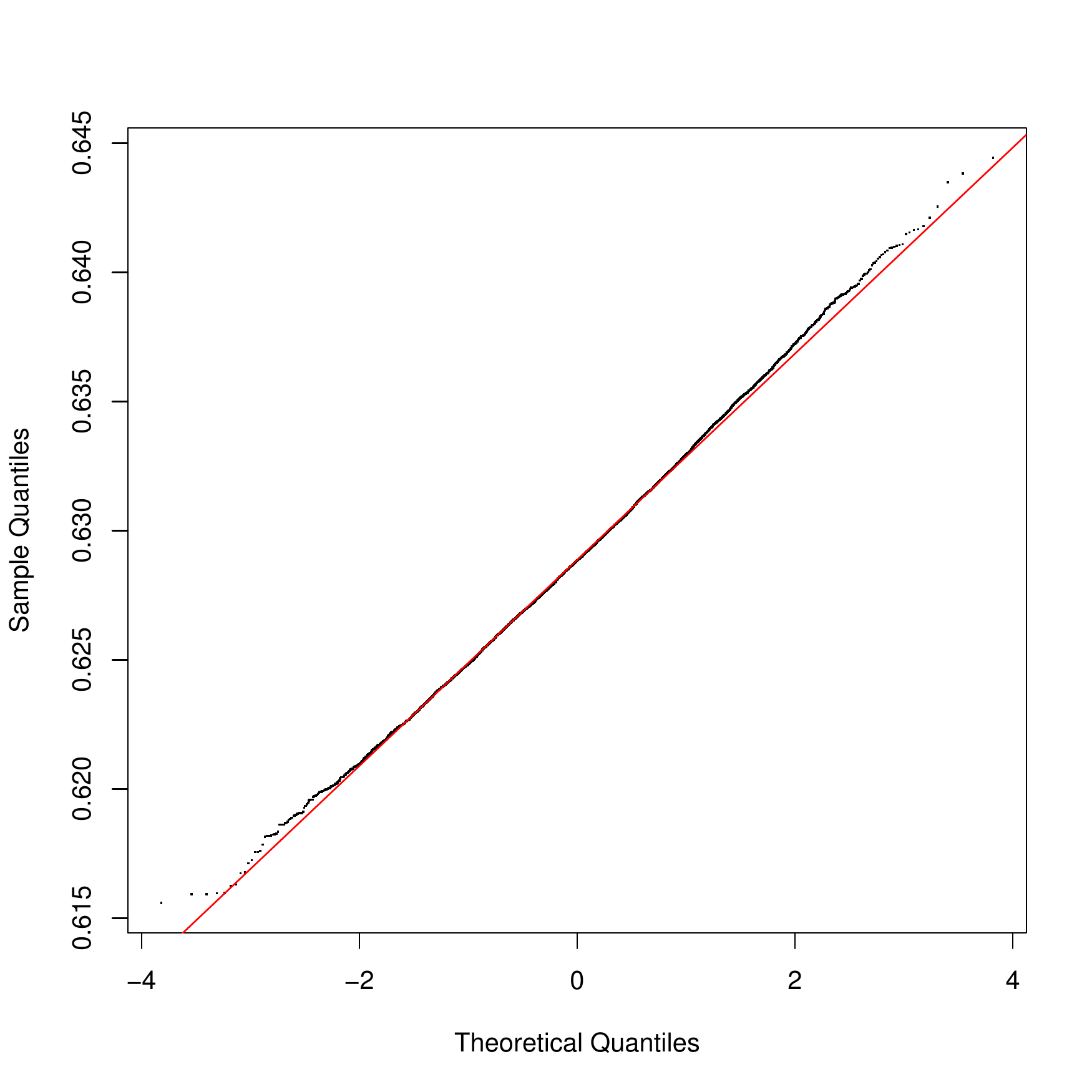}\includegraphics[height=5cm,width=5cm]{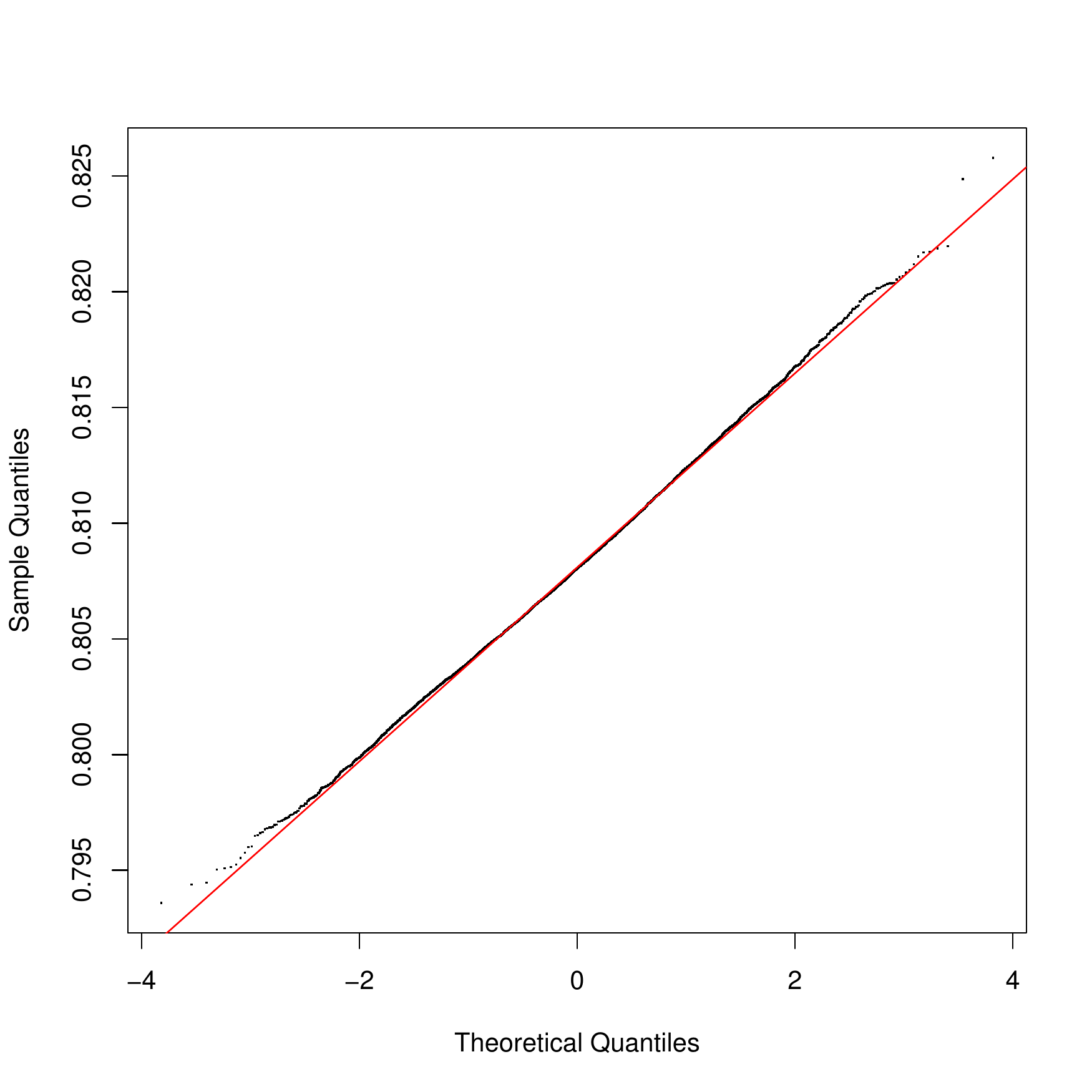}\\
(a) \hspace{4.3cm} (b) \hspace{4.3cm} (c)
\caption{\footnotesize{Q-Q plots for $\widetilde{D}_n$ vs Gaussian quantiles on Gross Income in panel (a),  on Net Income in panel (b) and on Taxes in panel (c)}}
\label{fig:Q-Q}
\end{figure}

\begin{figure}[h!]
\centering
\includegraphics[height=10cm]{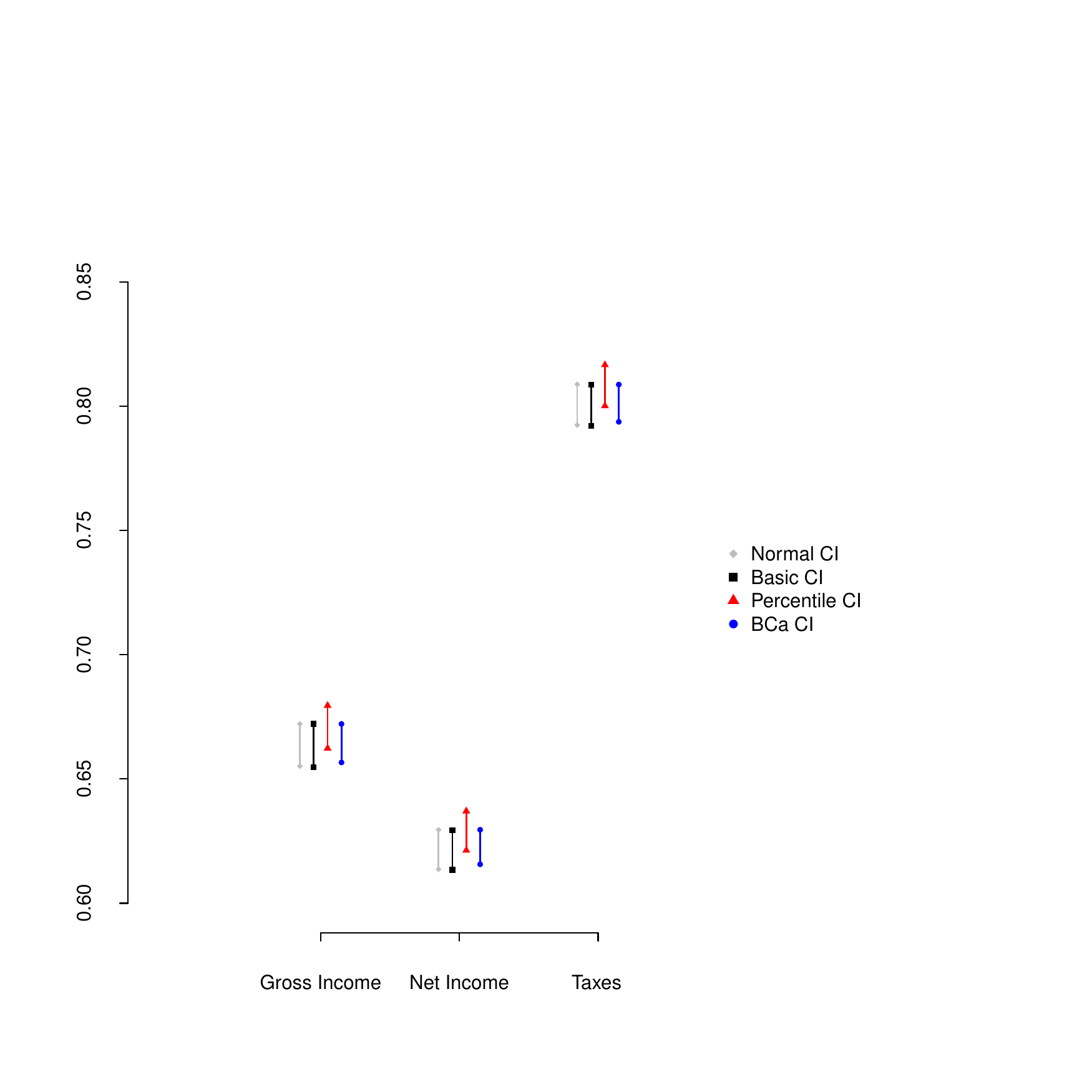}
\caption{\footnotesize{Confidence Intervals for $\widetilde{D}_n$ on Gross Income,  on Net Income and on Taxes}}
\label{fig:CI}
\end{figure}
Finally,  from Figure \ref{fig:CI} we observe that the four methods for constructing Confidence Intervals have a substantial agreement. They all agree in assuring that there is a statistically significant increase in inequality when passing from Net Income to Gross Income (the redistributive effect of  taxation) and from Gross Income to Taxes. We recall that adjusted percentile methods (also named BCa) for calculating confidence limits are inherently more accurate than the basic bootstrap and percentile methods (Davison and Hinkley, 1997).

\section{Concluding remarks}
\label{concRemarks}
Moved from the considerations that nowadays, in many  developed countries,  the more critical (i.e the extremes) portions of the population are facing great reshaping of their economic situation, a new index for measuring inequality have been proposed in Davydov and Greselin (2016). In the cited paper, a discussion of the properties of the index  has been given, to motivate its introduction in the literature and to show its descriptive features. Inferential results for the index were still missing, and this paper is a first contribution to fill the gap. After proposing two empirical estimators, we have shown their asymptotic equivalence. Then,   consistency and asymptotic normality for the first estimator have been derived. We also proved the convergence of the empirical estimator for the variance to its finite theoretical value. Finally, we used the new statistical inferential results to analyze data on Net Income from the Bank of Italy Survey on Household Income and Wealth, and to compare them with Gross Income and Taxes. 

%
%
%
%

\section*{References}

\def\hang{\hangindent=\parindent\noindent}
\rm
\hang
Atkinson A. B. 1970.
On the Measurement of Inequality,
\textit{Journal of Economic Theory}, 2, 244--263.

\hang
Cobham A., Sumner A. 2013.
Putting the Gini back in the bottle? The Palma as a Policy-Relevant Measure of Inequality,
\textit{Technical report}, Kings College London.

\hang
Davydov Yu., Zitikis R. 2004.
\textit{Convex rearrangements of random elements}
in Asymptotic Methods in
Stochastics, vol. 44 of Fields Institute Communications, pp. 141–171, American Mathematical Society,
Providence, RI, USA.

\hang
Davydov Yu., Greselin F. 2016.
Comparisons between poorest and richest to measure inequality, \textit{(under  revision)}

\hang
Davison A. C., Hinkley D. V. 1997. 
\textit{Bootstrap methods and their application (Vol. 1)}. Cambridge University Press.

\hang
Gastwirth J.L. 2016.
 Measures of Economic Inequality Focusing on the Status of the Lower and Middle Income Groups, \textit{Statistics and Public Policy}, 3:1, 1-9.
 
\hang
Gastwirth J.L. 2014.
Median-based measures of inequality: reassessing the increase in income inequality in the U.S. and Sweden, \textit{Journal of the International Association for Official Statistics}, 30, 311--320.

\hang
Gini C. 1914.
Sulla misura della concentrazione e della  variabilit{\`a} dei caratteri,
\textit{Atti del Reale Istituto Veneto di Scienze, Lettere ed Arti}, 73, 1203--1248. (English translation in Gini, C. 2005.
On the measurement of concentration and variability of characters.
\textit{Metron}, 63, 3--38.)

\hang
Greselin F. 2014. 
More equal and poorer, or richer but more unequal?,
\textit{Economic Quality Control}, 29, 2, 99--117.

\hang
Greselin F., Pasquazzi L., Zitikis R. 2010. 
Zenga's new index of economic inequality, its estimation, and an analysis of incomes in Italy, 
\textit{Journal of Probability and
Statistics}, 26 pp., DOI:10.1155/2010/718905.

\hang
Greselin F., Pasquazzi L., Zitikis R. 2012. 
Contrasting the Gini and Zenga indices
of economic inequality, 
\textit{Journal of Applied Statistics}, 40, 2, 282–297.

\hang 
Greselin F., Puri M. L.,  Zitikis R. 2009.
L-functions, processes, and statistics in measuring economic
inequality and actuarial risks, \textit{Statistics and Its Interface}, 2, 2, pp. 227–245.

\hang
Goldie C. M. 1977.
Convergence theorems for empirical Lorenz curves and their inverses
 \textit{ Advances in Applied Probability}, 9, 765--791.
 
\hang
Lorenz M.C. 1905.
Methods of measuring the concentration of wealth,
\textit{Journal of the American Statistical Association}, 9, 209--219.

\hang Morini M., Pellegrino S. 2016.
Personal income tax reforms: A genetic algorithm approach, 
\textit{European Journal of Operational Research}, first online, https://doi-org.proxy.unimib.it/10.1016/j.ejor.2016.07.059
\hang Pietra G.  1915.
Delle relazioni tra gli indici di variabilit{\`a}. Note I. \textit{Atti del Reale Istituto Veneto di Scienze,
Lettere ed Arti}, 74, 775--792 

\hang  
Pietra G. 2014.
On the relationships between variability indices. Note I. \textit{Metron}, 72, 5--16

\hang
Zenga M. 2007.
Inequality curve and inequality index based on the ratios between lower and upper arithmetic means,
\textit{Statistica \& Applicazioni}, 5, 3--27.

\hang
Zitikis R. 1998. The Vervaat process, in Szyszkowicz, B. (ed.), \textit{Asymptotic Methods in Probability and Statistics}, Elsevier Science, Amsterdam, pp. 667-694.

 \end{document}